\pdfoutput=1
\documentclass[12pt,a4paper]{article} 

\usepackage{amsmath}
\usepackage{amssymb}
\usepackage{graphicx}
\usepackage[ruled]{algorithm2e}
\usepackage{a4wide}
\usepackage{amsthm}

\newtheorem{theorem}{Theorem}[section]
\newtheorem{lemma}[theorem]{Lemma}
\newtheorem{corollary}[theorem]{Corollary}

\newcommand{\ie}{\emph{i.e.}}
\newcommand{\Tmax}{\ensuremath{\mbox{Tmax}}}

\newcommand{\fig}[3][1]{
\begin{figure}[!h]
\centerline{\scalebox{#1}{\includegraphics{#2.eps}}}
\caption{\label{fig_#2} \em #3} \end{figure}}

\renewcommand{\input}[1]{}

\begin{document}
\thispagestyle{empty}

\begin{center}
\LARGE Generalized Tilings
       with Height Functions\\
\vskip 0.2cm
\bf \large
Olivier Bodini\,\footnote{LIP, \'Ecole Normale Sup\'erieure de Lyon,
                          46, All\'ee d'Italie,
                          69364 Lyon Cedex 05, France.
                          \texttt{olivier.bodini@ens-lyon.fr}}
and
Matthieu Latapy\,\footnote{LIAFA, Universit\'e Paris 7,
                           2, place Jussieu,
                           75005 Paris, France.
                           \texttt{latapy@liafa.jussieu.fr}}
\end{center}
\vskip 0.5cm

\begin{flushright}
{\em ``Height functions are cool!''}\\
Cris Moore,
DM-CCG'01 talk \cite{LMN01}.
\end{flushright}
\vskip 0.5cm

\noindent
\textbf{Abstract:}
In this paper, we introduce a generalization of a class of tilings
which appear in the literature: the tilings over which a height
function can be defined (for example, the famous
tilings of polyominoes with dominoes).
We show that many properties of these tilings
can be seen as the consequences of
properties of the generalized tilings we introduce. In particular,
we show that any tiling problem which can be modelized in our
generalized framework has the following properties: the tilability
of a region can be constructively decided in polynomial time,
the number of connected components in the undirected flip-accessibility
graph can be determined, and the directed flip-accessibility graph
induces a distributive lattice structure.
Finally, we give a few examples of known tiling problems which
can be viewed as
particular cases of the new notions we introduce.

\vskip 0.2cm


\section{Preliminaries}

Given a finite set of elementary shapes, called \emph{tiles},
a \emph{tiling} of a given region is a set of
translated tiles such that the union of the tiles covers exactly
the region, and such that there is no overlapping between any
tiles. See for example Figure~\ref{fig_ex_domino_tiling} for a tiling
of a polyomino (set of squares on a two-dimensional grid) with
dominoes ($1\times 2$ and $2\times 1$ rectangles).
Tilings are widely used in physics to modelize natural objects
and phenomena. For example, quasicrystals are modelized by Penrose
tilings \cite{CF96} and dimers on a lattice are modelized by domino
tilings \cite{Ken00}.
Tilings appeared in computer science with the famous undecidability of
the question of whether the plane is tilable or not using a given finite
set of tiles \cite{Ber66}. Since then, many studies appeared concerning
these objects, which are also strongly related to many important
combinatorial problems \cite{Lat00}.

\fig[0.5]{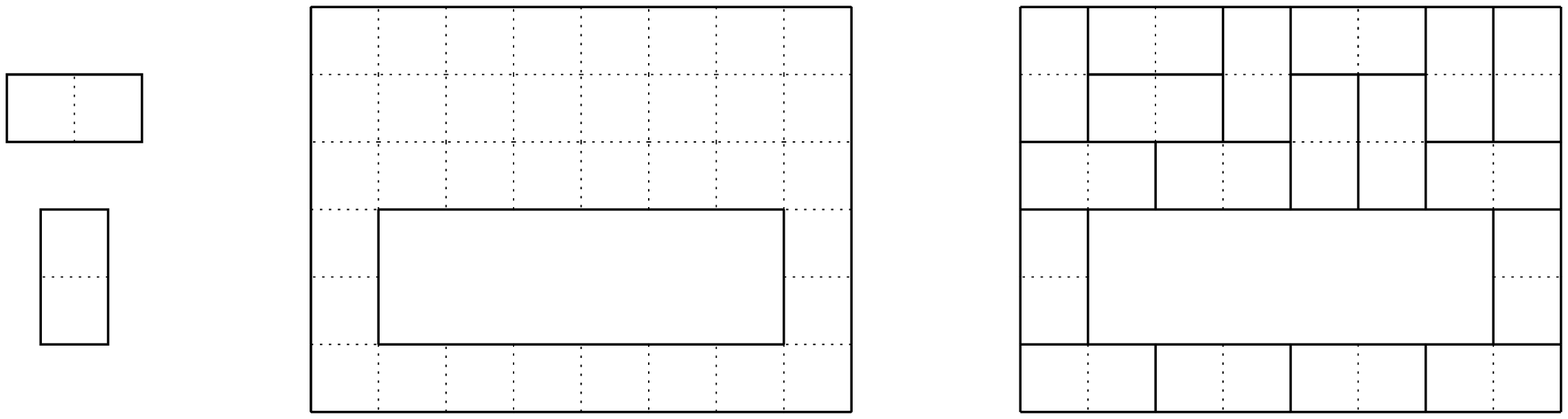}{From left to right: the two possible tiles (called
\emph{dominoes}), a polyomino (\ie\ a set of squares) to tile, and a
possible tiling of the polyomino with dominoes.}

A local
transformation is often defined over tilings. This transformation, called
\emph{flip}, is a local rearrangement of some tiles which
makes possible to obtain a new tiling from a given one.
One then defines the (undirected) \emph{flip-accessibility graph}
of the tilings of a region $R$,
denoted by $\overline{A_R}$, as follows:
the vertices of $\overline{A_R}$ are all the tilings of $R$,
and $\{t,t'\}$ is an (undirected) edge of $\overline{A_R}$ if and only if
there is a flip between $t$ and $t'$.
See Figure~\ref{fig_ex_domino_flip} for an example.
The flip notion is a key element for the enumeration of the
tilings of a given region, and for many algorithmical questions.
For example, we will see in the following that the structure of
$\overline{A_R}$ may give a way to sample randomly a tiling
of $R$, which is crucial
for physicists.
This notion is also a key element to study the entropy of the
physical object \cite{LR93}, and to examine
some of its
properties like frozen areas, weaknesses, and others \cite{JPS01}.

\fig[0.5]{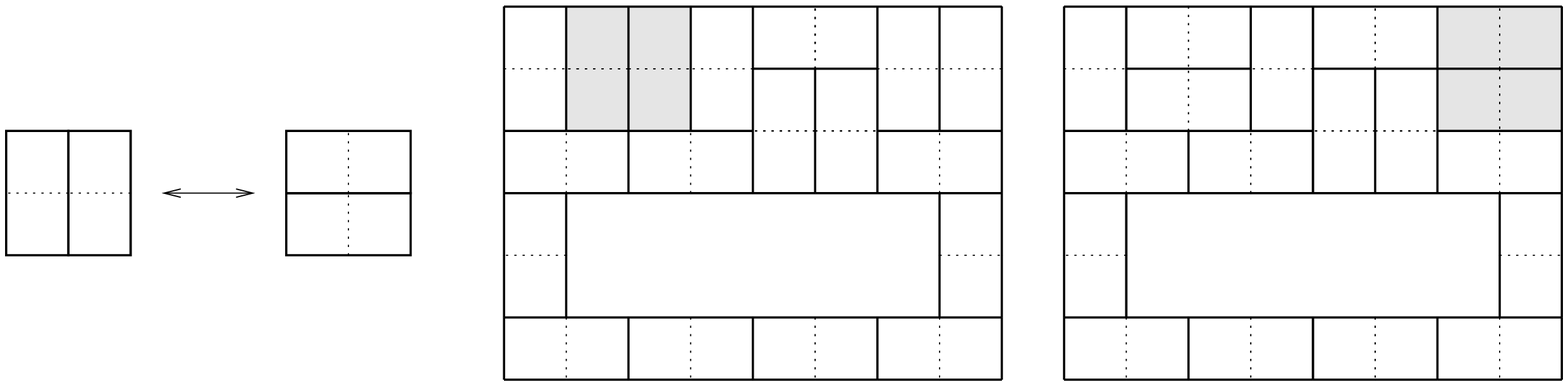}{From left to right: the flip operation
over dominoes, and two examples of tilings which can
be obtained from the one shown in Figure~\ref{fig_ex_domino_tiling}
by one flip. In these tilings, we shaded the tiles which moved during
the flip.}

On some classes of tilings which can be drawn on a regular grid,
it is possible to define a \emph{height
function} which associates an integer to any node of the grid (it
is called the \emph{height} of the point).
For example, one can define such a function over domino tilings as
follows. As already noticed, a domino tiling can be drawn on a two
dimensional square grid. We can draw the squares of the grid in black
and white like on a chessboard. Let us consider a polyomino $P$ and
a domino tiling $T$ of $P$, and let
us distinguish a particular point $p$ on the boundary of $P$, say
the one with smaller
coordinates. We say that $p$ is of height $0$, and that the height of
any other point $p'$ of $P$ is computed as follows: initialize a counter
to zero, and go from $p$ to
$p'$ using a path composed only of edges of dominoes in $T$,
increasing the counter when the square on the right is black and 
decreasing it when the square is white. The height of $p'$ is the
value of the counter when one reaches $p'$.
One can prove that this definition
is consistent and can be used as the height function for domino tilings
\cite{Thu90}. See
Figure~\ref{fig_ex_domino_height_graph} for an example.

\fig[0.5]{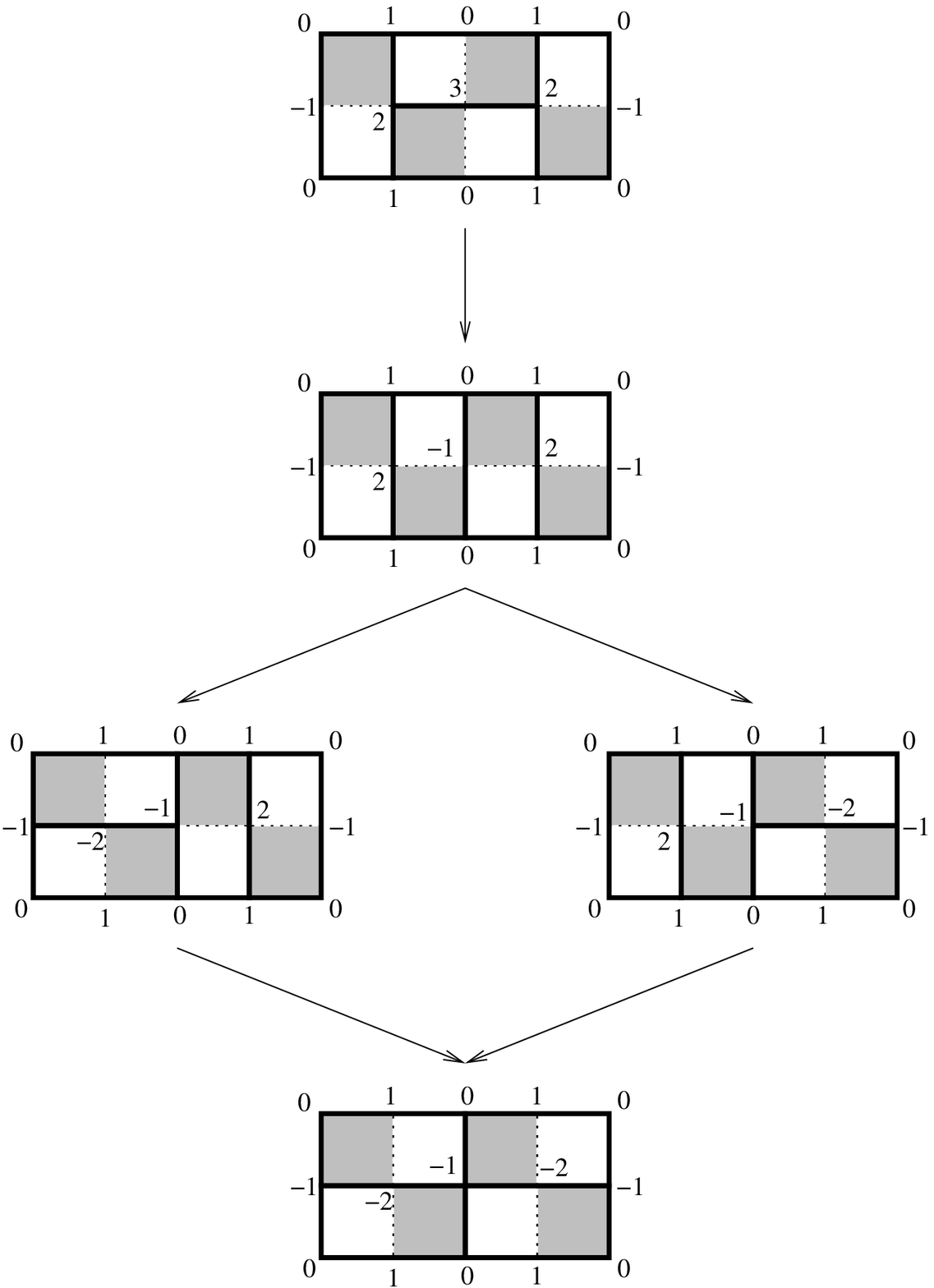}{The directed flip-accessibility graph of 
the tilings of a polyomino by dominoes. The height of each point 
of the polyomino is shown for each tiling. The set of all the tilings 
of this polyomino is ordered by the flip relation directed with respect
to the height functions.}

These height functions make it possible to define $A_R$, the \emph{directed}
flip-accessibility graph of the tilings of $R$: the vertices of $A_R$
are the tilings of $R$ and there is a directed edge $(t,t')$ if and only
if $t$ can be transformed into $t'$ by a flip which decreases the sum of
the heights of all the points. See Figure~\ref{fig_ex_domino_height_graph}
for an example with
domino tilings.
The generalized tilings we introduce in
this paper are based on these height functions, and most of our results
are induced by them.

These notions of height functions are close
to classical notions of flows theory in graphs.
Let $G=(V,E)$ be a directed graph. A {\em flow} on $G$ is a map
from $E$ into $\mathbb C$ (actually, we will only use flows with
values in $\mathbb Z$). Given two vertices $v$ and $v'$ of $G$,
a {\em travel} from $s$ to $s'$ is a set of edges of $G$ such that, if
one forgets their orientations, then one obtains a path from $s$ to $s'$.
Given a flow $C$, the {\em flux} of $C$ on the travel $T$ is
$$F_T(C)\ =\ \sum_{e\in T^+} C(e)\ -\ \sum_{e\in T^-} C(e)$$
where $T^+$ is the set of vertices of $T$ which are traveled
in the right direction when one goes from $s$ to $s'$, and $T^-$
is the set of vertices traveled in the reverse direction.
One can easily notice that the flux is additive by concatenation
of travels: 
if $T_1$ and $T_2$ are two travels such that the ending point of $T_1$ is equal
to the starting point of $T_2$, then
$F_{T_1\cdot T_2}(C) \ = \ F_{T_1}(C) + F_{T_2}(C)$.
See \cite{Ahu93} for more details about flows theory in graphs.

Since there is no circuit in the graph $A_R$ (there exists no nonempty sequence
of flips which transforms a tiling into itself), it induces an
order relation over all the tilings of $R$: $t \le t'$ if and only if
$t'$ can be obtained from $t$ by a sequence of (directed) flips. In
Section~\ref{sec_AP}, we will study $A_R$ under the order
theory point of view, and we will meet some special classes of orders,
which we introduce now.
A lattice is an order $L$ such that any two elements
$x$ and $y$ of $L$ have
a greatest lower bound, called the \emph{infimum} of $x$ and $y$
and denoted by $x \wedge y$, and a lowest greater bound, called
the \emph{supremum} of $x$ and $y$ and denoted by $x \vee y$.
The infimum of $x$ and $y$ is nothing but the greatest element
among the ones which are lower than both $x$ and $y$. The supremum
is defined dually.
A lattice $L$ is \emph{distributive} if for all $x$, $y$ and $z$ in $L$,
$x \vee (y \wedge z) = (x \vee y) \wedge (x \vee z)$ and
$x \wedge (y \vee z) = (x \wedge y) \vee (x \wedge z)$.
For example, it is known that the flip-accessibility graph
of the domino tilings of a polyomino without holes is always
a distributive lattice \cite{Rem99b}. Therefore, this is the case of
the flip-accessibility graph
shown in Figure~\ref{fig_ex_domino_height_graph} (notice that the
maximal element of the order is at the bottom, and the minimal one
at the top of the diagram since we used the discrete dynamical models
convention: the flips go from top to bottom).
Lattices (and especially distributive lattices)
are strongly structured sets.
Their study is an important part
of order theory, and many results about them exist. In particular,
various codings and algorithms are known about lattices and distributive
lattices. For example, there exists a generic algorithm to sample
randomly an element of any distributive lattice \cite{Pro98}.
For more details about orders and lattices, we refer to \cite{DP90}.

Finally, let us introduce a useful notation about graphs.
Given a directed graph $G=(V,E)$, the undirected graph
$\overline{G}=(\overline{V},\overline{E})$ is the graph
obtained from $G$ by removing the orientations of the
edges. In other words, $\overline{V}=V$, and $\overline{E}$
is the set of undirected edges $\{v,v'\}$ such that $(v,v') \in E$.
We will also call $\overline{G}$ the \emph{undirected version}
of $G$. Notice that this is consistent with our definitions
of $A_R$ and $\overline{A_R}$.

In this paper, we introduce a generalization of tilings
on which a height function can be defined,
and show how some known results may be understood in
this more general context.
All along this paper, like we did in
the present section, we will use the tilings with dominoes
as a reference to illustrate our definitions and results.
We used this unique example because it is very famous and
simple, and permits to give clear figures. We emphasize
however on the fact that our definitions and results are much
more general, as explained in the last section of the paper.

\section{Generalized tilings.}

In this section, we give all the definitions of the generalized
notions we introduce, starting from the objects we tile
to the notions of tilings, height functions, and flips.
The first definitions are very general, therefore we will
only consider some classes of the obtained objects, in
order to make the more specific notions (mainly height
functions and flips) relevant in this context.
However, the general objects introduced may be useful
in other cases.

Let $G$ be a simple ($G$ has no multiple edges, no loops,
and if $(v,v')$ is an edge then $(v',v)$ can not be an edge) directed graph.
We consider a set $\Theta$ of elementary circuits of $G$, which
we will call {\em cells}.
Then, a {\em polycell} is any set of cells in $\Theta$.
Given a polycell $P$, we call the edges of cells in $P$ the
\emph{edges of $P$}, and their vertices the \emph{vertices of $P$}.
A polycell $P$ is {\em $k$-regular} if and only if there exists
an integer $k$ such that
each cell of $P$ is a circuit of length $k$. 
Given a polycell $P$, the \emph{boundary} of $P$, denoted by
$\partial P$, is a (arbitrarily) distinguished set of
edges of $P$.
We say that a vertex of $P$ is {\em on the boundary} of $P$
if it is the extremity
of an edge in $\partial P$. 
A polycell $P$ is {\em full} if the undirected boundary
$\overline{\partial P}$
is connected.

Given an edge $e$ of $P$ which is not in $\partial P$,
we call the set of all the cells in $P$ which have $e$ in common a {\em tile}.
We will always suppose that, given any set of cells of $P$, they have at most
one edge in common.
A \emph{tiling} $Q$ of a polycell $P$ is then a partition of the set of
cells of $P$ into tiles. 
A polycell $P$ which admits at least a tiling $Q$ is {\em tilable}.
Notice that, if one considers a tiling $Q$, then one has a natural
bijection $\pi$ between the tiles of $Q$ and a set of edges of $G$:
if $t$ is a tile in $Q$, then $\pi(t)$ is nothing but the edge which
is in each of the cells which define $t$ (recall that
we have made the assumption
that this edge is unique).
The edges in $\pi(Q) = \{ \pi(t),\ t \in Q \}$ are called
the {\em tiling edges} of $Q$.
See Figure~\ref{fig_ex_polycell_strange} and
Figure~\ref{fig_ex_polycell_tiling} for some examples.
Notice that if we distinguish exactly one edge of each cell
of a polycell $P$, then the distinguished edges can be viewed
as the tiling edges of $P$. Indeed, each edge induces a tile
(the set of cells which have this edge in common),
and each cell is in exactly in one tile.

\fig[0.5]{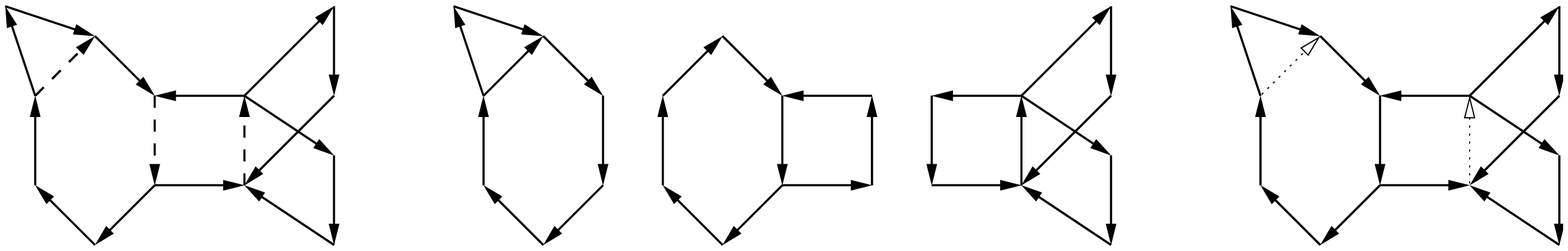}{From left to right: a polycell $P$
(the boundary $\partial P$ is composed of all the edges except the dotted ones),
the three tiles of $P$, and a tiling of $P$ represented by its tiling
edges (the dotted edges).}

\fig[0.7]{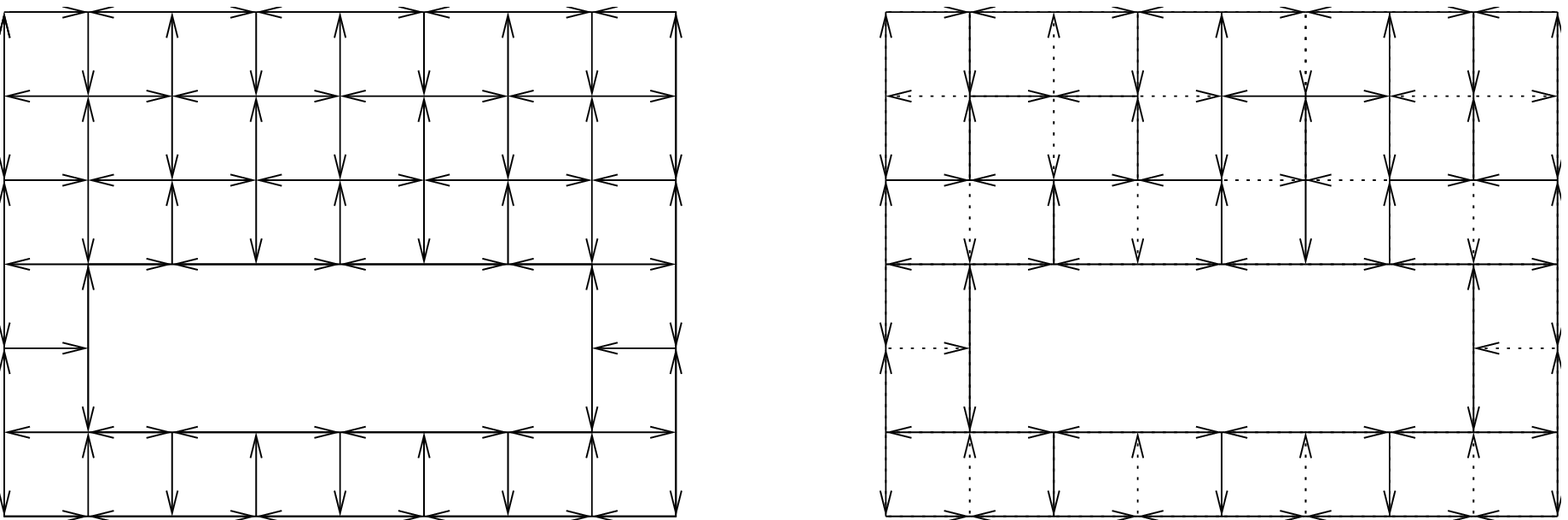}{Left: a $4$-regular polycell $P$, the boundary
of which is composed of those edges which belong to only
one cell. Right: a tiling of $P$
represented by its tiling edges (the dotted edges).
Notice that this figure is very close to
Figure~\ref{fig_ex_domino_tiling}.}

Let $P$ be a $k$-regular tilable polycell and $Q$ be a tiling of $P$.
We associate to $Q$ a flow $C_Q$ on $\Theta$ (seen as a graph):
$$
C_Q(e)= \left\{\begin{array}{ll}
                1-k & \mbox{if the edge $e$ is a tiling edge of $Q$}\\
                1   & \mbox{otherwise.}
               \end{array}\right.
$$
For each cell $c$, we define $T_c$ as the travel
which contains exactly the
edges in $c$ (in other words, it consists
in turning around $c$).
Notice that the flux of $C_Q$ on the travel $T_c$ is always null:
$F_{T_c}(C_Q)=0$ since each cell contains exactly a tiling edge,
valued $1-k$, and $k-1$ other edges, valued $1$.
Moreover, for each edge $e \in \partial P$,
we have $C_Q(e)=1$. 

Let us consider a polycell $P$ and a flow $C$ on the edges of $P$
such that $C(e)=1$ for all edge $e$ in $\partial P$.
If for all {\em closed} travel $T$ (\ie\ a cycle when one forgets the
orientation of each edge)
on the boundary of $P$ we have $F_T(C)=0$,
then we say that
$P$ has a {\em balanced boundary}. 
More specifically, if for all closed travel $T$ in $P$ (not only on the
boundary) we have $F_T(C)=0$, then the flow $C$ is called a {\em tension}.

\noindent
Finally, a polycell $P$ is {\em contractible} if it satisfies the two
following properties:
\begin{itemize}
\item $P$ has a balanced boundary.
\item $C$ is a tension if and only if for all cell $c$, $F_{T_c}(C)=0$.
\end{itemize}
Notice that if $P$ is a contractible $k$-regular polycell and $Q$
is a tiling of $P$, then
the flow $C_Q$ 
is a tension.

Now, if we (arbitrarily) distinguish a vertex $\nu$ on the
boundary of $P$, we can associate to the tension $C_Q$ a
{\em potential} $\varphi_Q$, defined over the
vertices of $P$:
\begin{itemize}
\item $\varphi_Q(\nu) = 0$.
\item for all vertices $x$ and $y$ of $P$,
      $\varphi_Q(y) - \varphi_Q(x) = F_{T_{x,y}}(C_Q)$
      where $T_{x,y}$ is a
      travel between $x$ and $y$.
\end{itemize}
The distinguished vertex is needed else
$\varphi_Q$ would only be defined at almost a constant,
but one can choose any vertex on the boundary.
Notice that this potential can be viewed as a \emph{height function}
associated to $Q$, and we will see that it indeed plays this role
in the following. Therefore, we will call the potential $\varphi_Q$
the \emph{height function} of $Q$.
See Figure~\ref{fig_ex_polycell_height} for an example.

\fig[0.7]{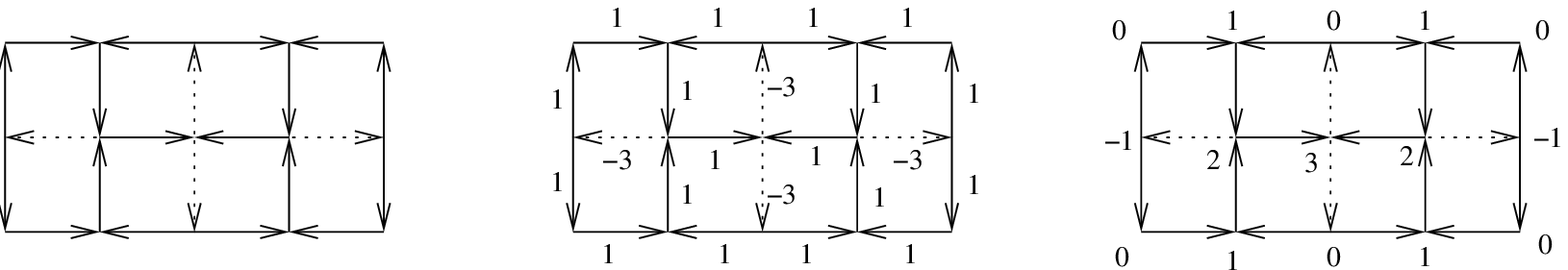}{From left to right:
a tiling $Q$ of a polycell (represented by
its tiling edges, the dotted ones), the tension $C_Q$ and the height function
(or potential) $\varphi_Q$ it induces. Again, this figure may be compared
to Figure~\ref{fig_ex_domino_height_graph} (topmost tiling).}

We now have all the main notions we need
about tilings of polycells,
including height functions, except the notion of flips.
In order to introduce it,
we need to prove the following:

\begin{theorem} \label{th:corres}
Let $P$ be a $k$-regular contractible polycell.
There is a bijection
between the tilings of $P$ and the tensions $C$
on $P$ which verify:
\begin{itemize}
\item for all edge $e$ in $\partial P$, $C(e)=1$,
\item and for all edge $e$ of $P$, $C(e) \in \{ 1-k, 1 \}$.
\end{itemize}
\end{theorem}
\begin{proof}
For all tiling $Q$ of $P$, we have defined above
a flow $C_Q$ which verifies the property in the claim, and
such that for all cell $c$, $F_{T_c}(C_Q)=0$. Since $P$
is contractible, this last point implies that
$C_Q$ is a tension.
Conversely, let us consider a tension $C$ which satisfies the hypotheses.
Since each cell is of length $k$, and since $C(e) \in \{ 1-k, 1 \}$,
the fact that $F_{T_c}(C)=0$ implies
that each cell has exactly one negative edge.
These negative edges can be considered as the
tiling edges of a tiling of $P$, which ends the proof.
\end{proof}

Given a $k$-regular contractible polycell $P$ defined over a graph $G$,
this theorem allows us to make no distinction between a tiling $Q$
and the associated tension $C_Q$. This makes it possible to define
the notion of flip as follows.
Suppose there is a vertex $x$ in $P$ which is not on the boundary and
such that its height, with respect to the height function
of $Q$, is greater than the height of each of its
neighbors in $\overline{G}$. We
will call such a vertex a \emph{maximal} vertex.
The neighbors of $x$ in $\overline{G}$
have a smaller height than $x$,
therefore the outgoing edges
of $x$ in $G$ are tiling edges of $Q$ and 
the incoming edges of $x$ in $G$
are not. 
Let us consider function $C_{Q'}$
defined as follows:
$$
C_{Q'}(e) =
  \left\{\begin{array}{ll}
   1-k    & \mbox{if $e$ is an outgoing edge of $x$}\\
   1      & \mbox{if $e$ is an incoming edge of $x$}\\
   C_Q(e) & \mbox{else.}
  \end{array}\right.
$$
Each cell $c$ which contains $x$
contains exactly one outgoing edge of $x$
and one incoming edge of $x$, therefore we still
have $F_{T_c}(C_{Q'}) = 0$.
Therefore, $C_{Q'}$ is a tension,
and so it induces from Theorem~\ref{th:corres} a tiling $Q'$.
We say that $Q'$ is obtained from $Q$ by a \emph{flip around $x$},
or simply by a \emph{flip}.
Notice that $Q'$ can also be defined as the tiling associated to the
height function obtained from the one of $Q$ by decreasing the height
of $x$ by $k$, and without changing anything else. This corresponds
to what happens with classical tilings (see for example \cite{Rem99b}).
See Figure~\ref{fig_ex_polycell_flip} for an example.

\fig[0.7]{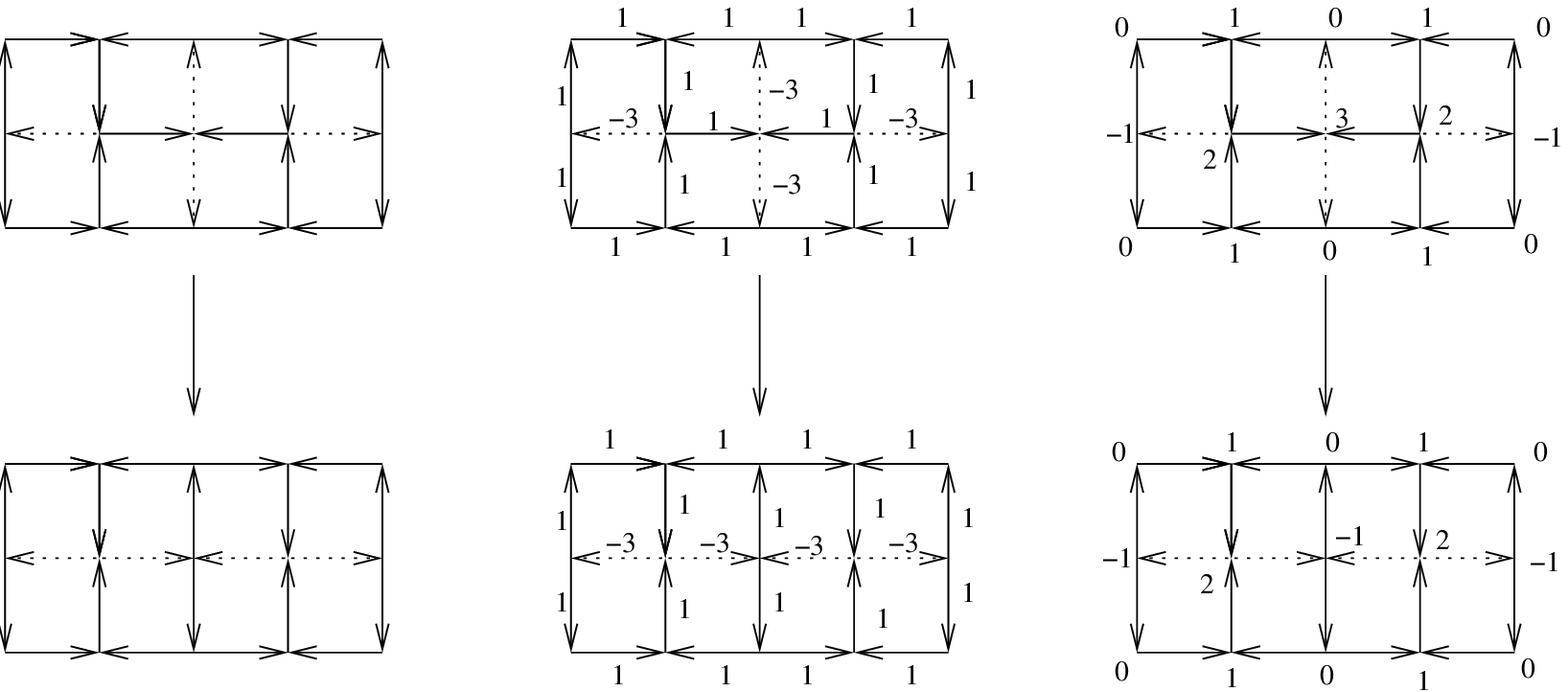}{A flip which transforms a tiling $Q$ 
of a polycell $P$ into
another tiling $Q'$ of $P$. From left to right, the flip is represented
between the tilings, then between the associated tensions, and finally
between the associated height functions.}

We now have all the material needed to define and study
$A_P$, the (directed) flip-accessibility graph of the tilings of $P$:
$A_P=(V_P,E_P)$ is the directed graph where $V_P$ is the set
of all the tilings of $P$ and 
$(Q,Q')$ is an edge in $E_P$ if $Q$ can be transformed into
$Q'$ by a flip.
We will also study the undirected flip-accessibility graph
$\overline{A_P}$.
The properties of these graphs are crucial for many questions
about tilings, like enumeration, generation and sampling.

\section{Structure of the flip-accessibility graph.}
\label{sec_AP}

Let us consider a $k$-regular contractible polycell $P$
and a tiling $Q$ of $P$. Let $h$ be the maximal value among
the heights of all the points with respect to the height
function of $Q$.
If $Q$ is such that all
the vertices of height $h$
are on the boundary of $P$, then it is said to be a \emph{maximal tiling}.
For a given $P$, we denote by $\Tmax_P$ the set of the 
maximal tilings of $P$.
We will see that these tilings play a particular role in the graph $A_P$.
In particular, we will give an explicit relation between them and
the number of connected components of $\overline{A_P}$.
Recall that we defined the maximal vertices of $Q$ as the vertices which
have a height greater than the height of each of their neighbors, with
respect to the height function of $Q$ (they are \emph{local} maximals).

\begin{lemma} \label{lm:maxinboundary}
Let $P$ be a $k$-regular tilable contractible polycell
($P$ is not necessarily full).
There exists a maximal tiling $Q$ of $P$.
\end{lemma}
\begin{proof}
Let $V$ be the set of vertices of $P$, and let
$Q$ be a tiling of $P$ such that for all tiling $Q'$ of $P$, we have:
$$\sum_{x\in V} \varphi_Q(x) \le
 \sum_{x\in V} \varphi_{Q'}(x).$$
We will prove that $Q$ is a maximal tiling.
Suppose there is a maximal vertex $x_m$ which is not on the boundary.
Therefore, one can transform $Q$ into $Q'$ by a flip around $x_m$.
Then $\sum_{x\in V} \varphi_{Q'}(x)\ =\ \sum_{x\in V} \varphi_Q(x)-k$,
which is in contradiction with the hypothesis.
\end{proof}

\begin{lemma} \label{lm:reach}
For all tiling $Q$ of a $k$-regular contractible polycell $P$,
there exists a unique tiling in $\Tmax_P$ reachable
from $Q$ by a sequence of flips.
\end{lemma}
\begin{proof}
It is clear that at least one tiling in $\Tmax_P$ can be reached
from $Q$ by a sequence of flips, since the flip operation decreases
the sum of the heights, and since we know from the proof
of Lemma~\ref{lm:maxinboundary}
that a tiling such that this sum is minimal is always
in $\Tmax_P$.
We now have to prove that the tiling in $\Tmax_P$ we obtain does not
depend on the order in which we flip around the successive maximal
vertices.
Since making a flip around a maximal point $x$ is nothing but decreasing
its height by $k$ and keeping the other values,
if we have two maximal vertices $x$ and $x'$ then it is equivalent
to make first the flip around $x$ and after the flip around $x'$ or the
converse.
\end{proof}

\begin{lemma} \label{lm:phi->tiling}
Let $P$ be a $k$-regular contractible and tilable polycell. 
A tiling $Q$ in $\Tmax_P$ is totally determined by the values
of $\varphi_Q$ on $\partial P$.
\end{lemma}
\begin{proof}
The proof is by induction over the number of cells in $P$.
Let $x$ be a maximal vertex for $\varphi_Q$ in $\partial P$.
For all outgoing edges $e$ of $x$,
$C_Q(e)=1-k$ (otherwise $\varphi(x)$ would not be maximal). 
Therefore, these edges
can be considered as
tiling edges, and determine some tiles of a tiling $Q$ of $P$.
Iterating this process, one finally obtains $Q$.
See Figure~\ref{fig_ex_algo} for an example.
\end{proof}

\begin{theorem} \label{th:zz}
Let $P$ be a $k$-regular contractible and tilable polycell. 
The number of connected components in $\overline{A_P}$ is equal to the cardinal
of $\Tmax_P$.
\end{theorem}
\begin{proof}
Immediate from Lemma~\ref{lm:reach}.
\end{proof}

This theorem is very general and can explain many results which appeared
in previous papers. We obtain for example the following corollary,
which generalizes the
one saying that any domino tiling of a
polyomino can be transformed into any other one by a sequence of
flips \cite{BNRR95}.

\begin{corollary}
Let $P$ be a full $k$-regular contractible and tilable polycell.
There is a unique element in $\Tmax_P$, which implies that
$\overline{A_P}$ is connected.
\end{corollary}
\begin{proof}
Since $\overline{\partial P}$ is connected, the heights of the points in
$\partial P$ are 
totally determined by the orientation of the edges of $\partial P$
and do not depend on any tiling $Q$. 
Therefore, from Lemma~\ref{lm:phi->tiling},
there is a unique tiling in $\Tmax_P$. 
\end{proof}

As a consequence, if $P$ is a full tilable and contractible polycell,
the height of a vertex $x$ on the boundary of $P$ is independent
of the considered tiling.
In the case of full polyominoes,
this restriction of $\varphi_Q$ to the boundary of $P$ is called
{\em height on the boundary}
\cite{Fou97} and has been introduced in \cite{Thu90}.
Notice that this height on the boundary can be defined
in the more general case where $P$ has a balanced boundary.

Notice also that the proof of
Lemma~\ref{lm:phi->tiling} gives an algorithm
to build the unique maximal tiling of any $k$-regular contractible
and tilable full polycell $P$,
since the height function on the boundary of $P$
can be computed without knowing any tiling of $P$. See
Algorithm~\ref{algo_tiling} and
Figure~\ref{fig_ex_algo}.
This algorithm gives in polynomial time a tiling of $P$ if it is tilable.
It can also be used to decide whether $P$ is tilable or not.
Therefore, it generalizes the result of Thurston \cite{Thu90} saying
that it can be decided in polynomial time if a given polyomino is
tilable with dominoes.

\begin{algorithm}
\SetVline
\In{A full $k$-regular contractible polycell $P$,
    its boundary $\partial P$ and a distinguished
    vertex $\nu$ on this boundary.}
\Out{An array \emph{tension} on integers indexed by the
     edges of $P$ and another one \emph{height}
     indexed by the vertices of $P$. The first gives the tension
     associated to the maximal tiling, and the second gives its
     height function.}
\begin{flushleft}
\Begin{
 $P' \leftarrow P$\;
 $\mbox{height}[\nu] \leftarrow 0$\;
 \ForEach{edge $e=(v,v')$ in $\partial P'$}{
  $\mbox{tension}[e] \leftarrow 1$\;
  }
 \ForEach{vertex $v$ on the boundary of $P'$}{
  Compute $\mbox{height}[v]$ using the values in $\mbox{tension}$\;
  }
 \Repeat{$P'$ is empty}{
  \ForEach{vertex $v$ on the boundary of $P'$ which has the minimal height among the heights of all the vertices on the boundary}{
   \ForEach{incoming edge $e$ of $v$}{
    $\mbox{tension}[e] \leftarrow 1-k$\;
    \ForEach{edge $e'$ in a cell containing $e$}{
     $\mbox{tension}[e'] \leftarrow 1$\;
     }
    }
   \ForEach{edge $e=(v,v')$ such that $\mbox{tension}[e]$ has newly
            be computed}{
    Compute $\mbox{height}[v]$ and $\mbox{height}[v']$ using the
    values in $\mbox{tension}$\;
    }
   }
  Remove in $P'$ the cells which contain a negative edge\;
  Compute the boundary of $P'$: it is composed of all the
  vertices of $P'$ which have a computed height\;
  }
 }
\end{flushleft}
\caption{\label{algo_tiling}Computation of the maximal tiling of a
full $k$-regular contractible polycell.}
\end{algorithm}

\fig[0.7]{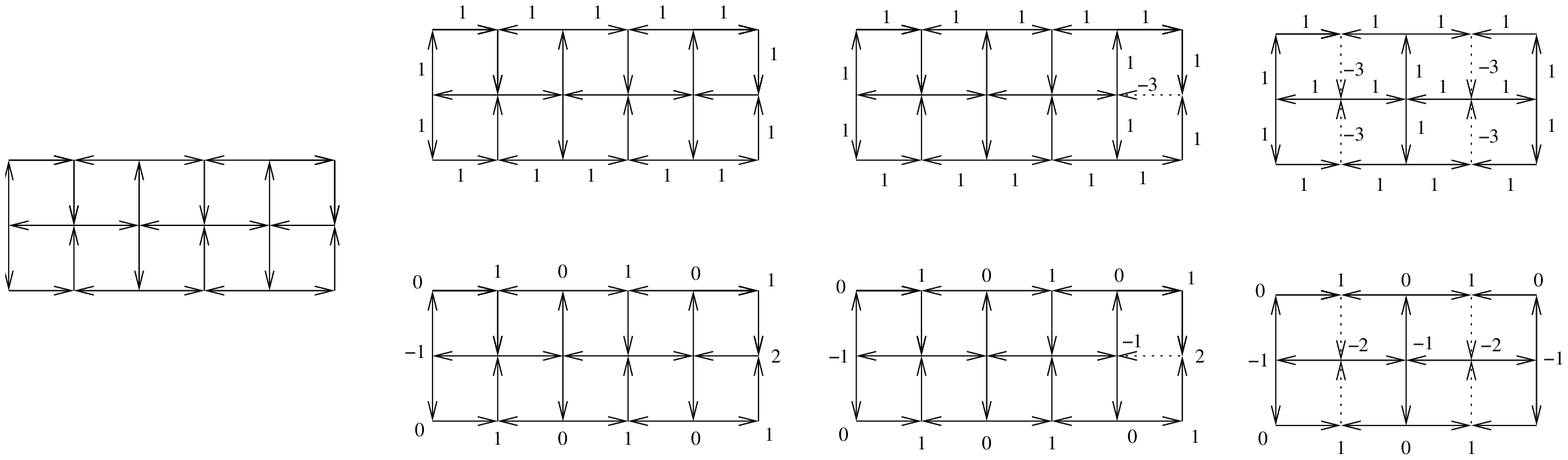}{An example of execution of Algorithm~\ref{algo_tiling}.
From left to right, we give the polycell, the result of the computation
of the height on the boundary, and then the results of each iteration
of the addition of tiles and removing of tiles process. In this
example, the first iteration of the algorithm gives one vertical
tile, and the second (and last) iteration gives four horizontal tiles.}

With these results, we obtained much information concerning a central
question of tilings: the connectivity of the undirected
flip-accessibility graph.
We did not only give a condition under which this graph is connected,
but we also gave a relation between the number of its connected components
and some special tilings. We will now deepen the study of the structure
induced by the flip relation by studying the directed flip-accessibility
graph, and in particular the partial order it induces over the tilings:
$t \le t'$ if and only if $t'$ can be obtained from $t$ by a sequence
of (directed) flips.

\begin{lemma} \label{lm:max}
Let $Q$ and $Q'$ be two tilings in the same connected component
of $A_P$ for a given $k$-regular contractible polycell $P$. 
Let us consider $x_m$ such that
$\left| \varphi_Q(x_m) - \varphi_{Q'}(x_m) \right|$ is maximal in
$\{ \left| \varphi_Q(x) - \varphi_{Q'}(x) \right|,\ x \mbox{ is a
vertex of $P$}\}$.
Then, one can make a flip around $x_m$ from $Q$ or $Q'$.
\end{lemma}
\begin{proof}
We can suppose that $\varphi_{Q'}(x_m) < \varphi_Q(x_m)$ (otherwise
we exchange $Q$ and $Q'$). 
We will show that the height function $\varphi$ defined by $\varphi(x_m) =
\varphi_Q(x_m) - k$ and $\varphi(x) = \varphi_Q(x)$ for all
vertex $x \not= x_m$ defines a tiling of $P$ (which is therefore
obtained from $Q$ by a flip around $x_m$).
Let us consider any circuit which contains $x_m$. Therefore,
it contains an incoming edge $(x_p,x_m)$ and an outgoing edge
$(x_m,x_s)$ of $x_m$.
We will prove that $\varphi_Q(x_p)=\varphi_Q(x_m)-1$ and
$\varphi_Q(x_s)=\varphi_Q(x_m)-k+1$, which will prove the claim
since it proves that $x_m$ is a maximal vertex.

The couple
$(\varphi_Q(x_p),\varphi_Q(x_s))$ can have three values:
$(\varphi_Q(x_m)-1,\varphi_Q(x_m)+1)$,
$(\varphi_Q(x_m)-1,\varphi_Q(x_m)-k+1)$,
or
$(\varphi_Q(x_m)+k-1,\varphi_Q(x_m)+1)$.
But, if $\varphi_Q(x_s)=\varphi_Q(x_m)+1$ then
$\varphi_{Q'}(x_s)=\varphi(x_m)+1$,
and so $\varphi_{Q'}(x_m)=\varphi(x_m)+k$, which is a contradiction.
If $\varphi_Q(x_p)=\varphi_Q(x_m)+k-1$ then
$\varphi_{Q'}(x_p)=\varphi_Q(x_m)+k-1$,
and so $\varphi_{Q'}(x_m)>\varphi_Q(x_m)$, which is a contradiction again.
Therefore, $(\varphi_Q(x_p),\varphi_Q(x_s))$ must be equal to
$(\varphi_Q(x_m)-1,\varphi_Q(x_m)-k+1)$ for all circuit which contain $x_m$,
which is what we needed to prove.
\end{proof}

Let us now consider two tilings $Q$ and $Q'$ of a $k$-regular
contractible polycell $P$.
Let us define $\max(\varphi_Q,\varphi_{Q'})$ as the height function
such that its value at each point is the maximal between
the values of $\varphi_Q$
and $\varphi_{Q'}$ at this point. Let us define
$\min(\varphi_Q,\varphi_{Q'})$ dually. Then, we have the following result:

\begin{lemma} \label{lem:}
Given two tilings $Q$ and $Q'$ of a $k$-regular
contractible polycell $P$,
$\max(\varphi_Q,\varphi_{Q'})$ and $\min(\varphi_Q,\varphi_{Q'})$
are the height functions of tilings
of $P$.
\end{lemma}
\begin{proof}
We can see that $\max(\varphi_Q,\varphi_{Q'})$ is the height function of
a tiling of $P$
by iterating
Lemma~\ref{lm:max}:
$\sum\limits_x \left| \varphi_Q(x)-\varphi_{Q'}(x) \right|$
can be decreased until $Q = Q'$.
The proof for $\min(\varphi_Q,\varphi_{Q'})$ is symmetric.
\end{proof}

\begin{theorem} \label{th:lattice}
If $P$ is a $k$-regular contractible polycell, then $A_P$
induces a distributive lattice structure over the tilings of $P$.
\end{theorem}
\begin{proof}
Given two tilings $Q$ and $Q'$ in $A_P$,
let us define the following binary operations:
$\varphi_Q\wedge\varphi_{Q'}=\min(\varphi_Q,\varphi_{Q'})$ and 
$\varphi_Q\vee\varphi_{Q'}=\max(\varphi_Q,\varphi_{Q'})$.
It is clear from the previous results that this defines
the infimum and supremum of $Q$ and $Q'$.
To show that the obtained lattice is \emph{distributive}, it suffices
now to verify that these relations are distributive together.
\end{proof}

As already discussed, this last theorem gives much information on the
structure of the flip-accessibility graphs of tilings of polycells.
It also gives the possibility to use in the context of tilings
the numerous results known about distributive
lattices, in particular the generic random sampling algorithm
described in \cite{Pro98}.

To finish this section, we give another proof of Theorem~\ref{th:lattice}
using only discrete dynamical models notions. This proof is very simple
and has the advantage of putting two combinatorial object in a relation which
may help understanding them. However, the reader not interested in
discrete dynamical models may skip the end of this section.

An Edge Firing Game (EFG) is defined by a connected undirected graph $G$ with
a distinguished vertex $\nu$, and an orientation
$O$ of $G$. In other words, $\overline{O}=G$. We then consider the set
of obtainable orientations when we iterate the following rule:
if a vertex $v\not= \nu$ only has incoming edges (it is a \emph{sink}) then
one can reverse all these edges. This set of orientations is ordered
by the reflexive and transitive closure of the evolution rule, and it is
proved in \cite{Pro93} that it is a distributive lattice. We will
show that the set of tilings of any $k$-regular contractible polycell $P$
Theorem~\ref{th:lattice}.

Let us consider a $k$-regular contractible
polycell $P$ defined over a graph $G$, and $G'$ the
sub-graph of $G$ which contains exactly the vertices and edges in $P$.
Let us now consider the height function $\varphi_Q$ of a tiling $Q$ of $P$,
and let us define the orientation $\pi(Q)$ of $\overline{G'}$
as follows:
each undirected edge $\{v,v'\}$ in $\overline{G'}$ is directed
from $v$ to $v'$ in $\pi(Q)$ if $\varphi_Q(v') > \varphi_Q(v)$.
Then, the maximal vertices of $Q$ are exactly the ones
which have only incoming edges in $\pi(Q)$, and applying the EFG rule
to a vertex of $\pi(Q)$ is clearly equivalent to making a flip around
this vertex in $Q$. Therefore, the configuration space of the EFG is
isomorphic to the flip-accessibility graph $A_P$, which proves
Theorem~\ref{th:lattice}.

\section{Some applications.}

In this section, we study some examples which appear in the
literature with the help of our generalized framework.
We show how these classes of tiling problems can be
seen as special cases of $k$-regular contractible
polycells tilings. We therefore
obtain as corollaries some known results about these problems,
as well as some new results.

\subsection{Polycell drawn on the plane.}

Let us consider a set of vertices $V$ and a set
$\Theta$ of elementary (undirected) cycles
of length $k$, with vertices in $V$, such that
any couple of cycles in $\Theta$ have at most one
edge in common.
Not let us consider the undirected graph $G=(V,E)$ such that
$e$ is an edge of $G$ if and only if it is an edge
of a cycle in $\Theta$. Moreover, let us restrict ourselves
to the case where $G$ is a planar graph
which can be drawn in such a way that
no cycle of $\Theta$ is drawn inside another one.
$G$ is 2-dual-colorable if one can color in black and white
each bounded face
in such a way that two faces which have an
edge in common have different colors. See for example
Figure~\ref{fig_ex_2_col}.

\fig[0.3]{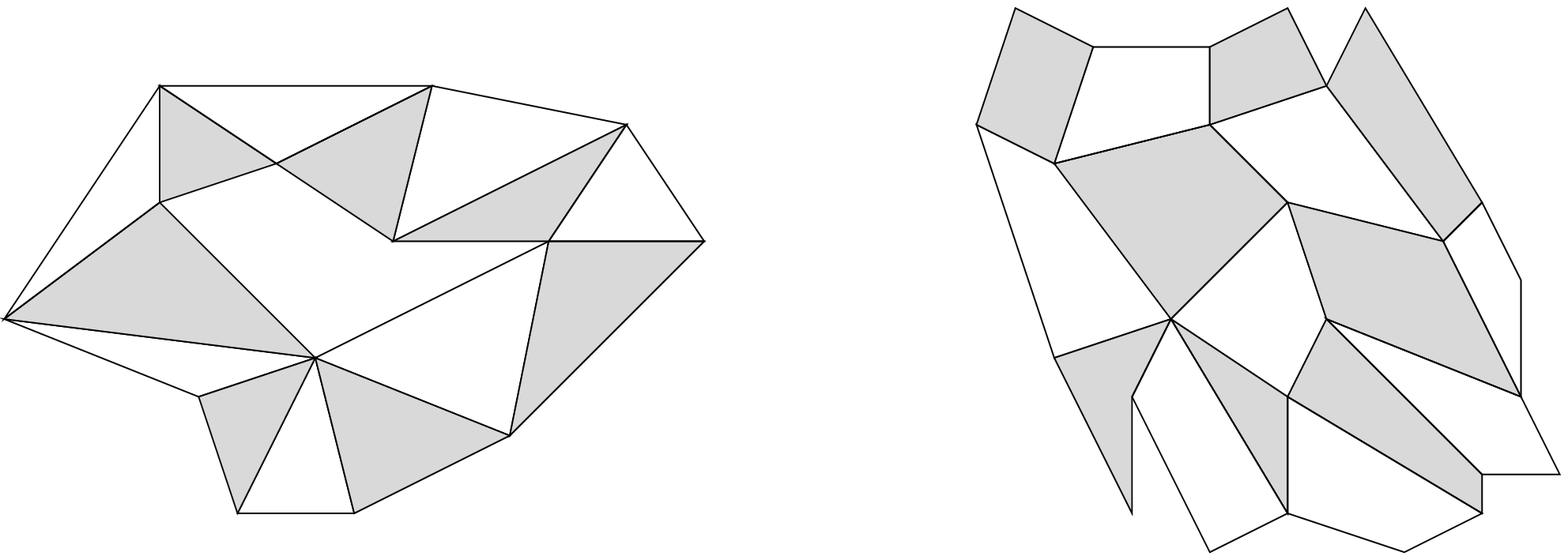}{Two examples of graphs which satisfy all the
properties given in the text. The leftmost is composed of cycles
of length $3$ and has a hole. The rightmost one is composed
of cycles of length $4$.}

\fig[0.3]{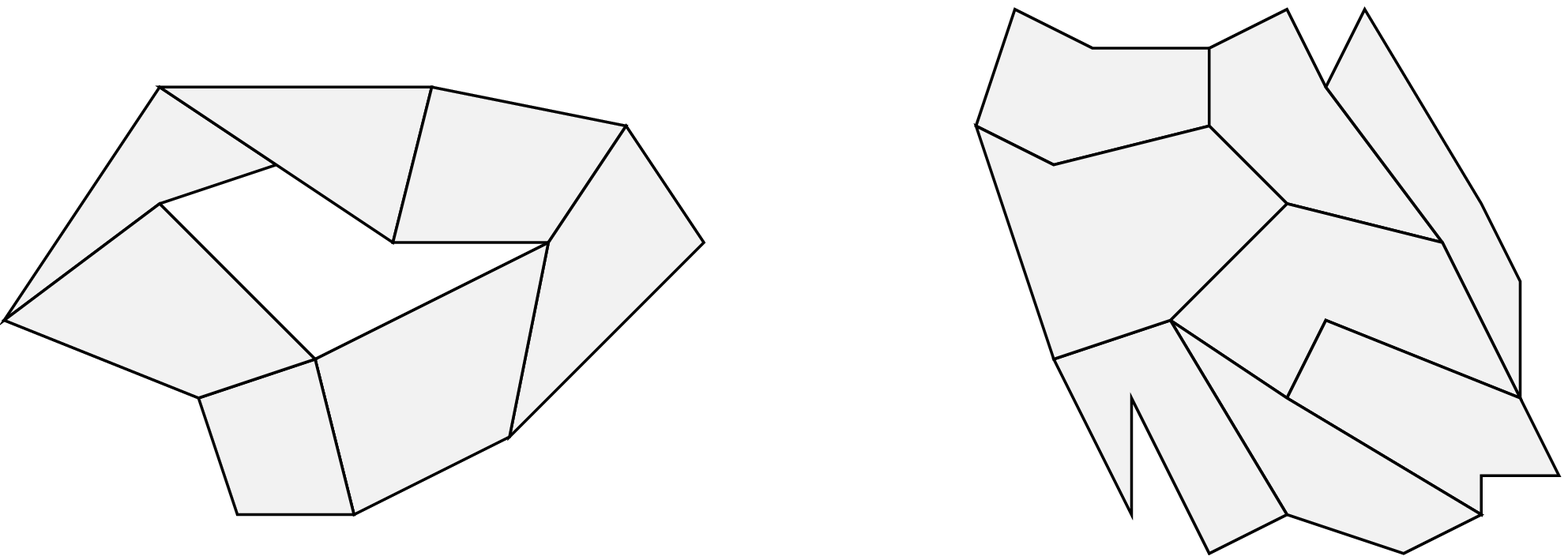}{A tiling of each of the objects
shown in Figure~\ref{fig_ex_2_col}, obtained using the
polycells formailsm.}

The fact that $G$ has the properties above, including being
2-dual-colorable, makes it possible to encode tilings with bifaces
(the tiles are two adjacent faces)
as tilings of polycells. This includes tilings with dominoes, and
tilings with calissons.
Following Thurston [Thu90], let us
define an oriented version of G as follows: the edges which constitute
the white cycles
boundaries are directed to travel the cycle in the clockwise
orientation, and the edges
which constitute the black cycles boundaries are directed
counterclockwise. One can
then verify that a balanced boundary polycell defined this way is always
contractible.
Therefore, our results can be applied, which generalizes the results
of Chaboud \cite{Cha96} and Thurston \cite{Thu90}.

\subsection{Rhombus tiling in higher dimension.}

Let us consider the canonical basis
$\{e_1,\dots,e_d\}$
of the $d$-dimensional affine space
${\mathbb R}^d$, and let us define
$e_{d+1} = \sum_{i=1}^d e_i$.
For all $\alpha$ between $1$ and $d+1$, let us define
the zonotope $Z_{d,d}^\alpha$ as the following set of points:
$$
Z_{d,d}^\alpha\ =\ \{ x\in {\mathbb R}^d
\mbox{ such that } x=\sum_{i=1,i\not= \alpha}^{d+1} \lambda_i e_i,
\mbox{ with } -1 \le \lambda_i \le 1 \}.
$$
In other words, the
$Z_{d,d}^\alpha$ is the zonotope defined by all the
vectors $\varepsilon_i$ except the $\alpha$-th.
We are interested in the tilability of a given solid $S$
when the set of allowed tiles is
$\{ Z_{d,d}^\alpha,\ 1 \le \alpha \le d+1\}$.
These tilings are called \emph{codimension one rhombus tilings},
and they are very important as a physical model of
quasicristals \cite{DMB97}. If $d=2$, they are nothing but the
tilings of regions of the plane with three parallelograms which
tile an hexagon, which have been widely studied.
See Figure~\ref{fig_ex_plane_grid} for an example in dimension
$2$, and Figure~\ref{fig_ex_til_zono_3d} for an example in
dimension $3$.

In order to encode this problem by a problem over polycells, let
us consider the directed graph $G$
with vertices in ${\mathbb Z}^d$ and 
such that $e=(x,y)$ is an edge if and only if
$y=x+\varepsilon_j$ for an integer $j$ between $1$ and $d$
or $y=x-\varepsilon_{d+1}$.
We will call {\em diagonal edges} the edges which correspond
to the second case. This graph can be viewed as a $d$-dimensional
directed grid (the direction are given by the order on the coordinates),
to which we add a diagonal edge in the reverse direction,
in each element of the grid.
An example in dimension $3$, is given
in Figure~\ref{fig_til_zono_3d_grid}.

\fig[0.3]{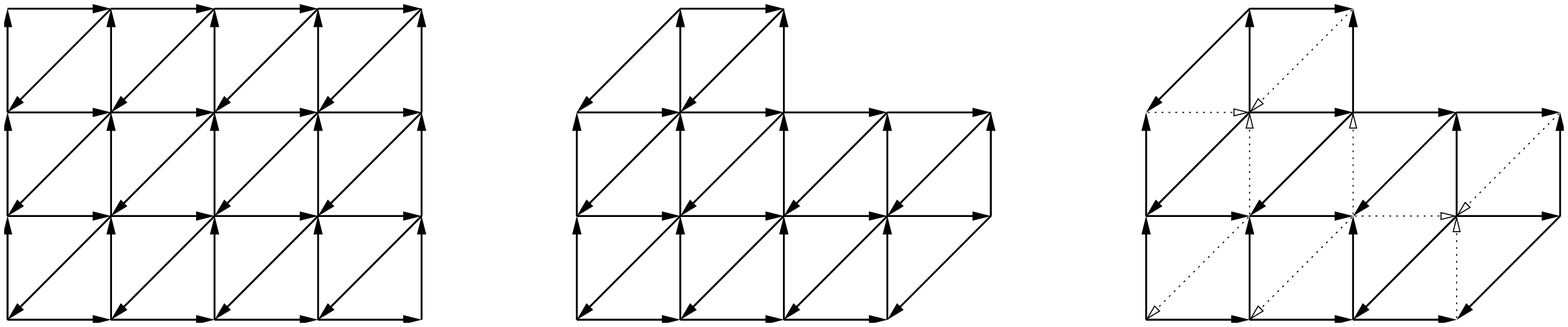}{If one forgets the
orientations and removes the dotted
edges, then the rightmost object is a classical codimencion one
rhombus tiling
of a part of the plane ($d=2$).
From the polycells point of view, the leftmost
object represents the underlying graph $G$, the middle object
represents a polucell $P$ (the boundary of which is the set of
the edges which belong to only one cell), and the rightmost object
represents a tiling of $P$ (the dottes edges are the tiling edges).}

\fig[0.3]{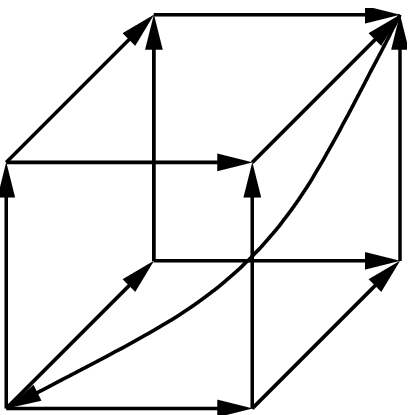}{The $3$-dimensional grid is obtained
by a concatenation of cubes like this one.}

\fig[0.5]{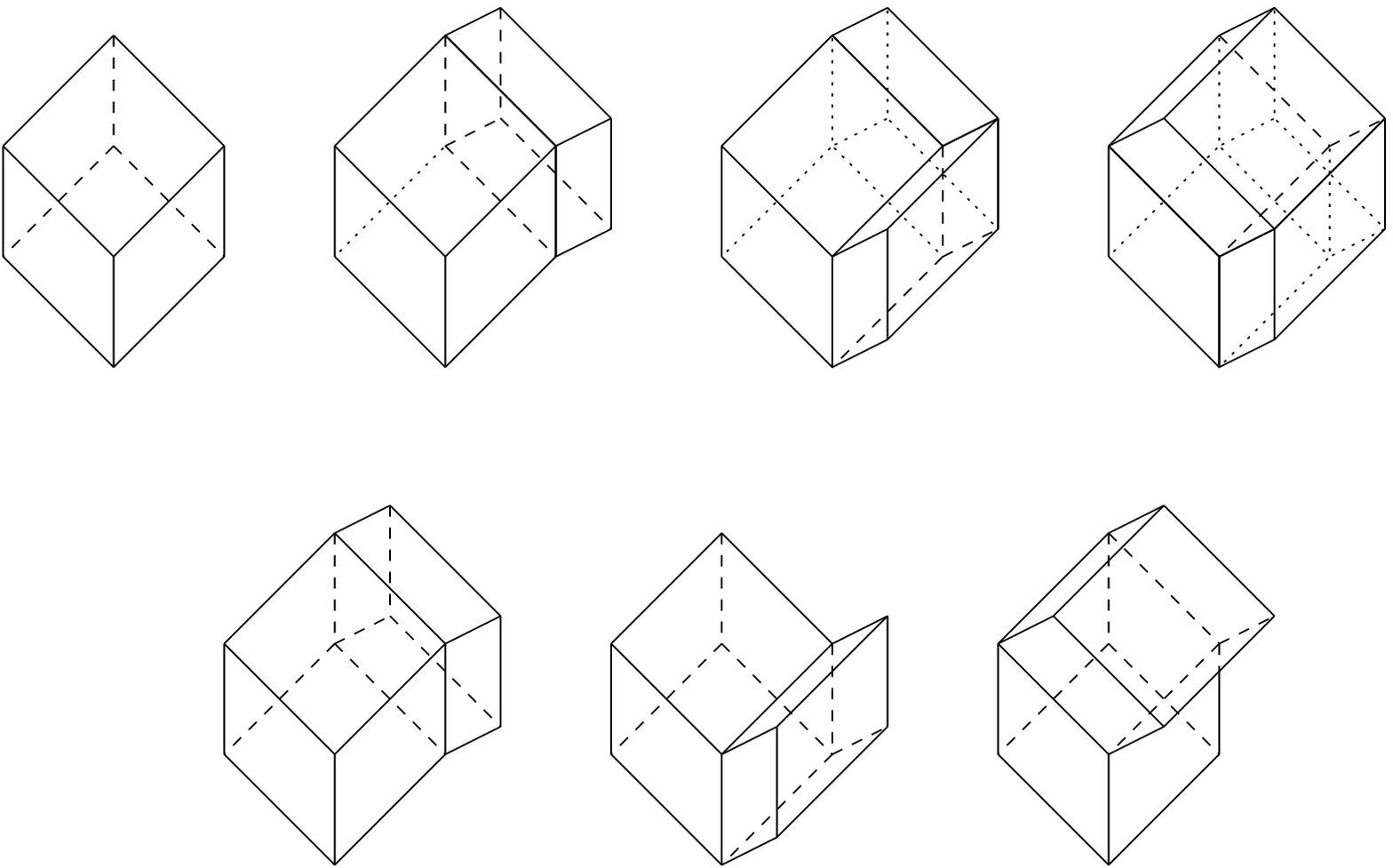}{A codimension one rhombus tiling
with $d=3$ (first line, rightmost object). It is composed
of four different three dimensional tiles, and the first
line shows how it can be constructing by adding successive
tiles. The second line shows the position of each tile with
respect to the cube.}

Each edge is clearly in a one-to-one correspondence with a
copy of a $Z_{d,d}^\alpha$ translated by an integer vector: this is
the copy on the $d$-dimensional grid of which it is a diagonal.
The set $\Theta$ of the cells we will consider is the
set of all the circuits of
length $d+1$ which contain exactly one diagonal edge. Therefore,
each edge belongs to a $d!$ cells, and so the tiles will
be themselves composed of $d!$ cells.
Given a polycell $P$ defined over $\Theta$, we define
$\partial P$ as the set of the edges of $P$ which do
not belong to $d!$ circuits of $P$.  

First notice that a full polycell defined over $G$
is always contractible.
Therefore, our previous results can be applied,
which generalizes some
results presented in \cite{DMB97} and \cite{LM99,LMN01}.
We also generalize some results about the 2-dimensional case,
which has been widely studied.

\section{Conclusion and Perspectives.}

In conclusion, we gave in this paper a generalized framework
to study the tiling problems over which a height function
can be defined. This includes the famous tilings of polyominoes
with dominoes, as well as various other classes, like
codimension one rhombus tilings, tilings on torus,
on spheres, three-dimensional tilings, and others we
did not detail here. We gave some
results on our generalized tilings which made it possible
to obtain a large set of known results as corollaries, as
well as to obtain new results on tiling problems which appear
in the scientific literature. Many other problems may exist
which can be modelized in the general framework we have introduced,
and we hope that this paper will help understanding them.

Many tiling problems, however, do not lead to the definition
of any height function. The key element to make such a function
exist is the presence of a strong underlying structure (the
$k$-regularity of the polycell, for example). Some important
tiling problems (for example tilings of zonotopes) do not have
this property, and so we can not apply our results in this context.
Some of these problems do not have the strong
properties we obtained on the tilings of $k$-regular contractible
polycells, but may be included in our framework,
since our basic definitions of polycells and tilings being very general.
This would lead to general results on more complex polycells,
for example polycells which are not $k$-regular.

\smallskip

\noindent \small
{\bf Acknowledgments:} The authors thank Fr\'ed\'eric Chavanon
for useful comments on preliminary versions, which deeply improved
the manuscript quality.

\bibliographystyle{alpha}
\bibliography{../Bib/bib.bib}

\end{document}